\documentclass[12pt]{amsart}
\usepackage{hyperref}
\usepackage{amssymb}

\newtheorem{theorem}{Theorem}[section]
\newtheorem{definition}[theorem]{Definition}

\newtheorem{proposition}[theorem]{Proposition}

\newtheorem{lemma}[theorem]{Lemma}

\begin{document}

\title{Classification results for easy quantum groups}

\author{T. Banica}
\address{T.B.: Department of Mathematics, Paul Sabatier University, 118 route de Narbonne, 31062 Toulouse, France. {\tt banica@math.ups-tlse.fr}}

\author{S. Curran}
\address{S.C.: Department of Mathematics, University of California, Berkeley, CA 94720, USA. {\tt curransr@math.berkeley.edu}}

\author{R. Speicher}
\address{R.S.: Department of Mathematics and Statistics, Queen's University, Jeffery Hall, Kingston, Ontario K7L 3N6, Canada. {\tt speicher@mast.queensu.ca}}


\begin{abstract}
We study the orthogonal quantum groups satisfying the ``easiness'' assumption axiomatized in our previous paper, with the construction of some new examples, and with some partial classification results. The conjectural conclusion is that the easy quantum groups consist of the previously known 14 examples, plus of an hypothetical multi-parameter ``hyperoctahedral series'', related to the complex reflection groups $H_n^s=\mathbb Z_s\wr S_n$. We discuss as well the general structure, and the computation of asymptotic laws of characters, for the new quantum groups that we construct.
\end{abstract}

\maketitle

\section*{Introduction}

One of the strengths of the theory of compact Lie groups comes from the fact that these objects can be classified. It is indeed extremely useful to know that the symmetry group of a classical or a quantum mechanical system falls into an advanced classification machinery, and applications of this method abound in mathematics and physics.

The quantum groups were introduced by Drinfeld \cite{dri} and Jimbo \cite{jim}, in order to deal with quite complicated systems, basically coming from number theory or quantum mechanics, whose symmetry group is not ``classical''. There are now available several extensions and generalizations of the Drinfeld-Jimbo construction, all of them more or less motivated by the same philosophy. A brief account of the whole story, focusing on constructions which are of interest for the present considerations, is as follows:

\begin{enumerate}
\item Let $G\subset U_n$ be a compact group, and consider the algebra $A=C(G)$. The matrix coordinates $u_{ij}\in A$ satisfy the commutation relations $ab=ba$. The original idea of Drinfeld-Jimbo, further processed by Woronowicz in \cite{wo1}, was that these commutation relations are in fact the $q=1$ particular case of the $q$-commutation relations $ab=qba$, where $q>0$ is a parameter. The algebra $A$ itself appears then as the $q=1$ particular case of a certain algebra $A_q$. While $A_q$ is no longer commutative, we can formally write $A=C(G_q)$, where $G_q$ is a quantum group.

\item An interesting modification of the above construction was proposed by Wang in \cite{wa1}, \cite{wa2}. His idea was to construct a new algebra $A^+$, by somehow ``removing'' the commutation relations $ab=ba$. Once again we can formally write $A^+=C(G^+)$, where $G^+$ is a so-called free quantum group. This construction, while originally coming only with a vague motivation from mathematical physics, was intensively studied in the last 15 years. Among the partial conclusions that we have so far is the fact that the combinatorics of $G^+$ is definitely interesting, and should have something to do with physics. In other words, $G^+$, while being by definition a quite abstract object, is probably the symmetry group of ``something'' very concrete. 

\item Several variations of Wang's construction appeared in the recent years, notably in connection with the construction and classification of intermediate quantum groups $G\subset G^*\subset G^+$. For instance in the case $G=O_n$, it was shown in our previous paper \cite{bsp} that the commutation relations $ab=ba$ can be succesfully replaced with the so-called half-commutation relations $abc=cba$, in order to obtain a new quantum group, $O_n^*$. Some other commutation-type relations, for instance of type $(ab)^s=(ba)^s$, will be described in the present paper.

\item As a conclusion, the general idea that tends to emerge from the above considerations is that a ``very large class'' of compact quantum groups should appear in the following way: (a) start with a compact Lie group $G\subset U_n$, (b) build a noncommutative version of $C(G)$, by replacing the commutation relations $ab=ba$ by some weaker relations, (c) deform this latter algebra, by using a positive parameter $q>0$, or more generally a whole family of such positive parameters.
\end{enumerate}

This was for the motivating story. In practice, now, while the construction (1) is now basically understood, thanks to about 25 years of efforts of many mathematicians, (2) is just at the very beginning of an axiomatization, (3) is still at the level of pioneering examples, and (4) is just a dream. As for the possible applications to physics, basically nothing is known so far, but the hope for such an application increases, as more and more interesting formulae emerge from the study of compact quantum groups.

The present paper is a continuation of our previous work \cite{bsp}. We will advance on the classification work started there, for the easy quantum groups in the orthogonal case, and we will present a detailed study of the new quantum groups that we find.

The objects of interest will be the compact quantum groups satisfying $S_n\subset G\subset O_n^+$. Here $O_n^+$ is the free analogue of the orthogonal group, constructed by Wang in \cite{wa1}, and for the compact quantum groups we use Woronowicz's formalism in \cite{wo1}.

As in \cite{bsp} we restrict attention to the ``easy'' case. The easiness assumption, essential to our considerations, roughly states that the tensor category of $G$ should be spanned by certain partitions, coming from the tensor category of $S_n$. This might look of course like a quite technical condition. The point, however, is that imposing this technical condition is the ``price to pay'' for restricting attention to the ``truly easy'' case.

As explained in \cite{bsp}, our motivating belief is the fact that ``any result which holds for $S_n,O_n$ should have a suitable extension to all easy quantum groups''. This is of course a quite vague statement, whose target is actually formed by some results at the borderline between representation theory and probability. In this paper, however, we will rather focus on the classification problem. The further development of our ``$S_n,O_n$ philosophy'', leading perhaps to some interesting applications, will be left to a number of forthcoming papers. We refer to the final section below for more comments in this direction.

So, for the purposes of the present work, the easy quantum groups can be just thought of as being a ``carefully chosen collection'' of basic objects of the theory.

There are 14 natural examples of easy quantum groups, all but one described in \cite{bsp}, and the remaining one to be studied in detail in this paper. In addition, there are at least two infinite series, once again to be introduced in this paper. The list is as follows:
\begin{enumerate}
\item Groups: $O_n,S_n,H_n,B_n,S_n',B_n'$.

\item Free versions: $O_n^+,S_n^+,H_n^+,B_n^+,S_n'^+,B_n'^+$.

\item Half-liberations: $O_n^*,H_n^*$.

\item Hyperoctahedral series: $H_n^{(s)},H_n^{[s]}$.
\end{enumerate}

This list doesn't cover all the easy quantum groups, but we will present here some partial classification results, with the conjectural conclusion that the full list should consist of (1,2,3), and of a multi-parameter series unifying (4). We will also investigate the new quantum groups that we find, by using various techniques from \cite{bb+}, \cite{bbc}, \cite{bsp}, \cite{bv1}, \cite{bv2}.

As already mentioned, we expect the above list to be a useful, fundamental ``input'' for a number of representation theory and probability considerations. We also expect that the new quantum groups that we find can lead in this way to some interesting applications. We have several projects here, to be discussed at the end of the paper.

The paper is organized as follows. In 1-2 we recall our previous results from \cite{bsp}, and we study the quantum group $H_n^*$, by using techniques from \cite{bsp}, \cite{bv2}. In 3-4 we introduce the one-parameter series, and we study their basic properties, by using techniques from \cite{bb+}, \cite{bv1}. In 5-6 we state and prove the classification results, by making a heavy use of the ``capping'' method in \cite{bsp}, \cite{bv2}. The final sections, 7-8, contain the computation of asymptotic laws of characters, and some concluding remarks.

\subsection*{Acknowledgements}

T.B. would like to thank Queen's University, where part of this work was done. The work of T.B. was supported by the ANR grants ``Galoisint'' and ``Granma'', and the work of R.S. was supported by a Discovery grant from NSERC.

\section*{0. Notation}

As in our previous work \cite{bsp}, the basic object under consideration will be a compact quantum group $G$. The concrete examples of such quantum groups include the usual compact groups $G$, and, to some extent, the duals of discrete groups $\widehat{\Gamma}$. In the general case, however, $G$ is just a fictional object, which exists only via its associated Hopf $C^*$-algebra of ``complex continuous functions'', denoted $A=C(G)$.

For simplicity of notation, we will rather use the quantum group $G$ instead of the Hopf algebra $A$. For instance $\int_Gu_{i_1j_1}\ldots u_{i_kj_k}\,du$ will denote the complex number obtained by applying the Haar functional $\varphi:A\to\mathbb C$ to the well-defined quantity $u_{i_1j_1}\ldots u_{i_kj_k}\in A$.

We will use the quantum group notation depending on the setting: in case where this can lead to confusion, we will rather switch back to the Hopf algebra notation.

\section{Easy quantum groups}

In this section we briefly recall some notions and results from our previous paper \cite{bsp}. This material is here mostly for fixing the formalism and the notations.

Consider first a compact group satisfying $S_n\subset G\subset O_n$. That is, $G\subset O_n$ is a closed subgroup, containing the subgroup $S_n\subset O_n$ formed by the permutation matrices.

Let $u,v$ be the fundamental representations of $G,S_n$. By functoriality we have an inclusion $Hom(u^{\otimes k},u^{\otimes l})\subset Hom(v^{\otimes k},v^{\otimes l})$, for any $k,l$. On the other hand, the Hom-spaces for $v$ are well-known: they are spanned by certain explicit operators $T_p$, with $p$ belonging to $P(k,l)$, the set of partitions between $k$ points and $l$ points. More precisely, if $e_1,\ldots,e_n$ denotes the standard basis of $\mathbb C^n$, the formula of $T_p$ is as follows:
$$T_p(e_{i_1}\otimes\ldots\otimes e_{i_k})=\sum_{j_1\ldots j_l}\delta_p\begin{pmatrix}i_1&\ldots&i_k\\ j_1&\ldots&j_l\end{pmatrix}e_{j_1}\otimes\ldots\otimes e_{j_l}$$

Here the $\delta$ symbol on the right is 0 or 1, depending on whether the indices ``fit'' or not, i.e. $\delta=1$ if all blocks of $p$ contains equal indices, and $\delta=0$ if not.

We conclude from the above discussion that the space $Hom(u^{\otimes k},u^{\otimes l})$ consists of certain linear combinations of operators of type $T_p$, with $p\in P(k,l)$.

We call $G$ ``easy'' if its tensor category is spanned by partitions.

\begin{definition}
A compact group $S_n\subset G\subset O_n$ is called easy if there exist sets $D(k,l)\subset P(k,l)$ such that $Hom(u^{\otimes k},u^{\otimes l})=span(T_p|p\in D(k,l))$, for any $k,l$.
\end{definition}

It follows from the axioms of tensor categories that the collection of sets $D(k,l)$ must be closed under certain categorical operations, notably the vertical and horizontal concatenation, and the upside-down  turning. The corresponding algebraic structure formed by the sets $D(k,l)$, axiomatized in \cite{bsp}, will be called ``category of crossing partitions''.

We recall that a matrix is called monomial if it has exactly one nonzero entry on each row and each column. The basic examples are the permutation matrices.

\begin{definition}
We consider the following groups:
\begin{enumerate}
\item $O_n$: the orthogonal group.

\item $S_n$: the symmetric group, formed by the permutation matrices.

\item $H_n$: the hyperoctahedral group, formed by monomial matrices with $\pm 1$ entries.

\item $B_n$: the bistochastic group, formed by orthogonal matrices with sum $1$ on each row.

\item $S_n'=\mathbb Z_2\times S_n$: the group formed by the permutation matrices times $\pm 1$.

\item $B_n'=\mathbb Z_2\times B_n$: the group formed by the bistochastic matrices times $\pm 1$.
\end{enumerate}
\end{definition}

It follows from definitions that all the above 6 groups satisfy $S_n\subset G\subset O_n$. Observe that among all these groups, only $O_n,S_n$ are of ``irreducible'' nature, because we have canonical isomorphisms $H_n=\mathbb Z_2\wr S_n$ and $B_n\simeq O_{n-1}$. See \cite{bsp}.

The partitions in $P(k,l)$ with $k+l$ even are called ``even''.

\begin{theorem}
There are exactly $6$ easy groups, namely the above ones. The corresponding categories of crossing partitions are as follows:
\begin{enumerate}
\item $P_o$: all pairings.

\item $P_s$: all partitions.

\item $P_h$: partitions with blocks of even size.

\item $P_b$: singletons and pairings.

\item $P_{s'}$: all partitions (even part).

\item $P_{b'}$: singletons and pairings (even part).
\end{enumerate}
\end{theorem}

This result is proved in \cite{bsp}. The idea is that the second assertion follows from some well-known results regarding the groups $O_n,S_n$ and their versions, and the first assertion can be proved by carefully manipulating the categorical axioms.

Let us discuss now the free analogue of the above results. Let $O_n^+,S_n^+$ be the free orthogonal and symmetric quantum groups, corresponding to the Hopf algebras $A_o(n),A_s(n)$ constructed by Wang in \cite{wa1}, \cite{wa2}. Here, and in what follows, we use Woronowicz's Hopf algebra formalism in \cite{wo1}, and its subsequent quantum group interpretation.

We have $S_n\subset S_n^+$, so by functoriality the Hom-spaces for $S_n^+$ appear as subspaces of the corresponding Hom-spaces for $S_n$. The Hom-spaces for $S_n^+$ have in fact a very simple description: they are spanned by the operators $T_p$, with $P\in NC(k,l)$, the set of noncrossing partitions between $k$ upper points and $l$ lower points.

We have the following ``free analogue'' of Definition 1.1.

\begin{definition}
A compact quantum group $S_n^+\subset G\subset O_n^+$ is called free if there exist sets $D(k,l)\subset NC(k,l)$ such that $Hom(u^{\otimes k},u^{\otimes l})=span(T_p|p\in D(k,l))$, for any $k,l$.
\end{definition}

In this definition, the word ``free'' has of course a quite subtle meaning, to be fully justified later on. For the moment, let us just record the fact that the passage from Definition 1.1 to Definition 1.4 is basically done by ``restricting attention to the noncrossing partitions'', which, according to \cite{spe}, should indeed lead to freeness.

As in the classical case, the sets of partitions $D(k,l)$ must be stable under certain categorical operations, coming this time from the axioms in \cite{wo2}. The corresponding algebraic structure, axiomatized in \cite{bsp}, is called ``category of noncrossing partitions''.

We denote by $H_n^+$ the hyperoctahedral quantum group, constructed in \cite{bbc}, and by $B_n^+,S_n'^+,B_n'^+$ the free analogues of the groups $B_n,S_n',B_n'$, constructed in \cite{bsp}. 

It is useful to recall at this point the definition of all quantum groups involved.

\begin{definition}
We consider the following quantum groups, all given with the defining relations between the basic coordinates $u_{ij}\in C(G)$:
\begin{enumerate}
\item $O_n^+$: orthogonality ($u_{ij}=u_{ij}^*$, $u^t=u^{-1}$).

\item $S_n^+$: magic condition (all rows and columns of $u$ are partitions of unity). 

\item $H_n^+$: cubic condition (orthogonality, and $u_{ij}u_{ik}=u_{ji}u_{ki}=0$ for $j\neq k$).

\item $B_n^+$: bistochastic condition (orthogonality, and on each row the sum is $1$).

\item $S_n'^+$: cubic condition, with the same sum on rows and columns.

\item $B_n'^+$: orthogonality, with the same sum on rows and columns.
\end{enumerate}
\end{definition}

Perhaps the very first observation is that for any of the groups $G$ appearing in Definition 1.2 we have $C(G)=C(G^+)/I$, where $I\subset C(G^+)$ is the commutator ideal. In other words, $G^+$ is indeed a ``noncommutative version'' of $G$. We refer to \cite{bsp} and to its predecessors \cite{bbc}, \cite{wa1}, \cite{wa2} for the whole story, and for a careful treatement of all this material.

We have the following ``free analogue'' of Theorem 1.3.

\begin{theorem}
There are exactly $6$ free quantum groups, namely the above ones. The corresponding categories of noncrossing partitions are as follows:
\begin{enumerate}
\item $NC_o$: all noncrossing pairings.

\item $NC_s$: all noncrossing partitions.

\item $NC_h$: noncrossing partitions with blocks of even size.

\item $NC_b$: singletons and noncrossing pairings.

\item $NC_{s'}$: all noncrossing partitions (even part).

\item $NC_{b'}$: singletons and noncrossing pairings (even part).
\end{enumerate}
\end{theorem}

Once again, this result is proved in \cite{bsp}. The idea is that the second assertion follows from some standard results regarding $O_n^+,S_n^+$ and their versions $H_n^+,B_n^+,S_n'^+,B_n'^+$, and the first assertion can be proved by carefully manipulating the categorical axioms.

Observe the symmetry between Theorem 1.3 and Theorem 1.6: this corresponds to the ``liberation'' operation for orthogonal Lie groups, further investigated in \cite{bsp}.

\section{Half-liberation}

We consider now the general situation where we have a compact quantum group satisfying $S_n\subset G\subset O_n^+$. Once again, we can ask for the tensor category of $G$ to be spanned by certain partitions, coming from the tensor category of $S_n$.

\begin{definition}
A compact quantum group $S_n\subset G\subset O_n^+$ is called easy if there exist sets $D(k,l)\subset P(k,l)$ such that $Hom(u^{\otimes k},u^{\otimes l})=span(T_p|p\in D(k,l))$, for any $k,l$.
\end{definition}

As a first remark, this definition generalizes at the same time Definition 1.1 and Definition 1.4. Indeed, the easy quantum groups $S_n\subset G\subset O_n^+$ satisfying the extra assumption $G\subset O_n$ are precisely the easy groups, and those satisfying the extra assumption $S_n^+\subset G$ are precisely the free quantum groups. This follows indeed from definitions, see \cite{bsp}.

Once again, the sets of partitions $D(k,l)$ must be stable under certain categorical operations, coming from Woronowicz's axioms in \cite{wo2}. The corresponding algebraic structure, axiomatized in \cite{bsp}, will be here simply called ``category of partitions''.

We already know that the easy quantum groups include the 6 easy groups in Theorem 1.3, and the 6 free quantum groups in Theorem 1.6. In general, the world of easy quantum groups is quite rigid, but we can produce some more examples in the following way.

\begin{definition}
The half-liberated version of an easy group $G$ is the quantum group $G^*$ given by $C(G^*)=C(G^+)/I$, where $I$ is the ideal generated by the half-commutation relations $abc=cba$, imposed to the basic matrix coordinates $u_{ij}\in C(G^+)$.
\end{definition}

In other words, instead of removing the commutativity relations of type $ab=ba$ from the standard presentation of $C(G)$, which would produce the algebra $C(G^+)$, we replace these commutativity relations by the weaker relations $abc=cba$.

In order to study the half-liberated versions, we need a categorical interpretation of the half-commutation relations $abc=cba$. Let us agree that the upper points of a partition $p\in P(k,l)$ are labeled $1,2,\ldots,k$, and the lower points are labeled $1',2',\ldots,l'$. 

With this notation, we have the following key lemma, from \cite{bsp}.

\begin{lemma}
For a compact quantum group $G\subset O_n^+$, the following are equivalent:
\begin{enumerate}

\item The basic coordinates $u_{ij}$ satisfy $abc=cba$.

\item We have $T_p\in End(u^{\otimes 3})$, where $p=(13')(22')(3'1)$.
\end{enumerate}
\end{lemma}

\begin{proof}
According to the definition of $T_p$ given in section 1, we have $T_p(e_a\otimes e_b\otimes e_c)=e_c\otimes e_b\otimes e_a$. This gives the following formulae:
\begin{eqnarray*}
T_pu^{\otimes 3}(e_a\otimes e_b\otimes e_c)&=&\sum_{ijk}e_k\otimes e_j\otimes e_i\otimes u_{ia}u_{jb}u_{kc}\\
u^{\otimes 3}T_p(e_a\otimes e_b\otimes e_c)&=&\sum_{ijk}e_i\otimes e_j\otimes e_k\otimes u_{ic}u_{jb}u_{ka}
\end{eqnarray*}

The identification of the right terms gives the equivalence in the statement.
\end{proof}

Let us go back now to the quantum groups $G^*$. Observe first that we have inclusions $G\subset G^*\subset G^+$. As pointed out in \cite{bsp}, the cases $G=S_n,B_n,S_n',B_n'$ are not interesting, because here we have $G=G^*$. This can be checked by a direct computation with generators and relations, or with the partition $p$ appearing in Lemma 2.3, and will follow as well from the general classification results in sections 5-6 below.

In the cases $G=O_n,H_n$, however, we obtain new quantum groups. Let us agree that the legs of each partition are labeled $1,2,3,\ldots$, clockwise starting from top left. 

\begin{theorem}
The half-liberated versions of $O_n,H_n$ are easy quantum groups, and the corresponding categories of partitions are as follows:
\begin{enumerate}
\item $E_o$: pairings with each string connecting an odd number to an even number.

\item $E_h$: partitions with each block having the same number of odd and even legs.
\end{enumerate}
\end{theorem}

\begin{proof}
Our claim is that $E_o,E_h$ are categories of partitions, corresponding respectively to the quantum groups $O_n^*,H_n^*$. Indeed, this can be checked as follows:

(1) Here $E_o$ is nothing but the set of pairings which each string having an even number of crossings, and the result was proved in \cite{bsp}. The idea is that $E_o$ is generated in the categorical sense by the partition $p$ appearing in Lemma 2.3.

(2) The fact that $E_h$ is indeed a category of partitions follows from definitions: the idea is to think of each block as being ``balanced'' with respect to the odd and even labels, and with this interpretation in mind, the categorical operations clearly preserve the balancing. For instance when checking the stability under composition, which is the crucial axiom, the point is that given a connected union of blocks of the two partitions which are composed, the ``balancing in the middle'' is subject to a cancelling.

As for the fact that $E_h$ corresponds to the above quantum group $H_n^*$, this can be checked in several ways. Consider for instance the following diagram:
$$\begin{matrix}
O_n^*&\subset&O_n^+\\
\\
\cup&&\cup\\
\\
H_n^*&\subset&H_n^+
\end{matrix}$$

We know from definitions that $H_n^*$ is obtained by putting together the relations for $O_n^*$ and for $H_n$, so we have the quantum group equality $H_n^*=O_n^*\cap H_n^+$. Now by the general properties of Tannakian duality, it follows that the category of partitions of $H_n^*$ is generated by the category of partitions for $H_n^+$, namely the noncrossing partitions having even blocks, and by the half-liberation partition $p$ in Lemma 2.3. 

Now this category is by definition included into $E_h$, and the reverse inclusion can be checked as well, by a straightforward computation.
\end{proof}

The quantum group $O_n^*$, which first appeared in \cite{bsp}, was further investigated in \cite{bv2}. In order to get some insight into the structure of $H_n^*$, we will use similar methods.

\begin{definition}
The projective version of a quantum group $G\subset U_n^+$ is the quantum group $PG\subset U_{n^2}^+$, having as basic coordinates the elements $v_{ij,kl}=u_{ik}u_{jl}^*$.
\end{definition}

In other words, $C(PG)\subset C(G)$ is the algebra generated by the elements $v_{ij,kl}=u_{ik}u_{jl}^*$. In the case where $G$ is a classical group we recover of course the well-known formula $PG=G/(G\cap T)$, where $T\subset U_n$ are the unitary diagonal matrices. We refer to \cite{bv2} for a full discussion of this notion, including a list of concrete examples. 

Consider now the compact group $K_n=\mathbb T\wr S_n$ consisting of monomial (i.e. permutation-like) matrices, with elements on the unit circle $\mathbb T$ as nonzero entries. 

The following result, whose first assertion is from \cite{bv2}, will play a key role in the study of $H_n^*$, and of the other quantum groups to be introduced in this paper.

\begin{theorem}
The projective versions of half-liberations are as follows:
\begin{enumerate}
\item $PO_n^*=PU_n$.

\item $PH_n^*=PK_n$.
\end{enumerate}
\end{theorem}

\begin{proof}
The first assertion is proved in \cite{bv2}, the idea being that the partitions for $PO_n^*$ and for $PU_n$ are the same. For the second assertion, we use a similar method. Observe first that from $H_n^*\subset O_n^*$ we get $PH_n^*\subset PO_n^*=PU_n$, so $PH_n^*$ is indeed a classical group.

In order to compute this group, consider the following diagram:
$$\begin{matrix}
K_n&\subset&U_n^+\\
\\
\cup&&\cup\\
\\
H_n&\subset&H_n^*
\end{matrix}$$

We fix $k,l\geq 0$ and we consider the formal words $\alpha=(u\otimes\bar{u})^{\otimes k}$ and $\beta=(u\otimes\bar{u})^{\otimes l}$. Our claim is that the corresponding spaces $Hom(\alpha,\beta)$ for our 4 quantum groups appear as span of the operators $T_p$, with $p$ belonging to the following 4 sets of partitions:
$$\begin{matrix}
E_h(2k,2l)&\supset&E_o(2k,2l)\\ 
\\ 
\cup&&\cup\\
\\
P_h(2k,2l)&\supset&E_h(2k,2l)
\end{matrix}$$

Indeed, the bottom left set is the good one, thanks to Theorem 1.3. The bottom right set is also the good one, thanks to Theorem 2.4. For the top right set, this follows from the equality $PO_n^*=PU_n$ and from Theorem 2.4, and for full details here we refer to \cite{bv2}. As for the top left set, this follows for instance from the various results in \cite{ba2}, \cite{bb+}, \cite{bv1} regarding $K_n^+$, after ``adding a crossing''. A direct proof here can be obtained as well, by working out the categorical interpretation of the various relations defining $K_n$.

Summarizing, we have computed the relevant diagrams for the projective versions of our four algebras. So, let us look now at these projective versions:
$$\begin{matrix}
PK_n&\subset&PU_n^+\\
\\
\cup&&\cup\\
\\
PH_n&\subset&PH_n^*
\end{matrix}$$

The quantum groups $PH_n^*,PK_n$ appear as subgroups of the same quantum group, namely $PU_n^+$, and the above discussion tells us that these subgroups have the same diagrams. By using the same argument as in \cite{bv2}, we conclude that we have $PH_n^*=PK_n$.
\end{proof}

\section{The hyperoctahedral series}

In this section introduce a new series of quantum groups, $H_n^{(s)}$ with $s\in\{2,3,\ldots,\infty\}$. These will ``interpolate'' between $H_n^{(2)}=H_n$ and $H_n^{(\infty)}=H_n^*$.

The quantum group $H_n^{(s)}$ is obtained from $H_n^*$ by imposing the ``$s$-commutation'' condition $abab\ldots=baba\ldots$ (length $s$ words) to the basic coordinates $u_{ij}$. It is convenient to write down the complete definition of $H_n^{(s)}$, which is as follows.

\begin{definition}
$C(H_n^{(s)})$ is the universal $C^*$-algebra generated by $n^2$ self-adjoint variables $u_{ij}$, subject to the following relations:
\begin{enumerate}
\item Orthogonality: $uu^t=u^tu=1$, where $u=(u_{ij})$ and $u^t=(u_{ji})$.

\item Cubic relations: $u_{ij}u_{ik}=u_{ji}u_{ki}=0$, for any $i$ and any $j\neq k$.

\item Half-commutation: $abc=cba$, for any $a,b,c\in \{u_{ij}\}$.

\item $s$-mixing relation: $abab\ldots=baba\ldots$ (length $s$ words), for any $a,b\in\{u_{ij}\}$.
\end{enumerate}
\end{definition}

The fact that $H_n^{(s)}$ is indeed a quantum group follows from the elementary fact that the cubic relations are of ``Hopf type'', i.e. that they allow the construction of the Hopf algebra maps $\Delta,\varepsilon,S$. This can be checked indeed by a routine computation.

Observe that at $s=2$ the $s$-mixing is the usual commutation $ab=ba$. This relation being stronger than the half-commutation $abc=cba$, we are led to the algebra generated by $n^2$ commuting self-adjoint variables satisfying (1,2), which is $C(H_n)$.

As for the case $s=\infty$, the $s$-mixing relation disappears here by definition. Thus we are led to the algebra defined by the relations (1,2,3), which is $C(H_n^*)$.

Summarizing, we have $H_n^{(2)}=H_n$ and $H_n^{(\infty)}=H_n^*$, as previously claimed. In what follows we present a detailed study of $H_n^{(s)}$, our first technical result being as follows.  

\begin{lemma}
For a compact quantum group $G\subset H_n^*$, the following are equivalent:
\begin{enumerate}

\item The basic coordinates $u_{ij}$ satisfy $abab\ldots=baba\ldots$ (length $s$ words).

\item We have $T_p\in End(u^{\otimes s})$, where $p=(135\ldots 2'4'6'\ldots)(246\ldots 1'3'5'\ldots)$.
\end{enumerate}
\end{lemma}

\begin{proof}
According to the definition of $T_p$ given in section 1, the operator associated to the partition in the statement is given by the following formula:
$$T_p(e_{a_1}\otimes e_{b_1}\otimes e_{a_2}\otimes e_{b_2}\otimes\ldots)=\delta(a)\delta(b)e_b\otimes e_a\otimes e_b\otimes e_a\otimes\ldots$$

Here we use the convention $\delta(a)=1$ if all the indices $a_i$ are equal, and $\delta(a)=0$ if not, along with a similar convention for $\delta(b)$. As for the indices $a,b$ appearing on the right, these are the common values of the $a$ indices and $b$ indices, respectively, in the case $\delta(a)=\delta(b)=1$, and are irrelevant quantities in the remaining cases.

This gives the following formulae:
\begin{eqnarray*}
T_pu^{\otimes s}(e_{a_1}\otimes e_{b_1}\otimes e_{a_2}\otimes\ldots)&=&\sum_{ij}e_i\otimes e_j\otimes e_i\otimes\ldots\otimes u_{ia_1}u_{jb_1}u_{ia_2}\ldots\\
u^{\otimes s}T_p(e_{a_1}\otimes e_{b_1}\otimes e_{a_2}\otimes\ldots)&=&\delta(a)\delta(b)\sum_{ij}e_{i_1}\otimes e_{j_1}\otimes e_{i_2}\otimes\ldots\otimes u_{i_1b}u_{j_1a}u_{i_2b}\ldots
\end{eqnarray*}

Here the upper sum is over all indices $i,j$, and the lower sum is over all multi-indices $i=(i_1,\ldots,i_s),j=(j_1,\ldots,j_s)$. The identification of the right terms, after a suitable relabeling of indices, gives the equivalence in the statement.
\end{proof}

Let us work out now the $s$-analogue of Theorem 2.4.

\begin{theorem}
$H_n^{(s)}$ is an easy quantum group, and its associated category $E_h^s$ is that of the ``$s$-balanced'' partitions, i.e. partitions satisfying the following conditions:
\begin{enumerate}
\item The total number of legs is even.

\item In each block, the number of odd legs equals the number of even legs, modulo $s$.
\end{enumerate}
\end{theorem}

\begin{proof}
As a first remark, at $s=2$ the first condition implies the second one, so here we simply get the partitions having an even number of legs, corresponding to $H_n$. Observe also that at $s=\infty$ we get the partitions which are balanced, in the sense of the proof of Theorem 2.4, which are known from there to correspond to the quantum group $H_n^*$.

Our first claim is that $E_h^s$ is indeed a category. But this follows simply by adapting the $s=\infty$ argument in the proof of Theorem 2.4, just by adding ``modulo $s$'' everywhere.

It remains to prove that this category corresponds indeed to $H_n^{(s)}$. But this follows from the fact that the partition $p$ appearing in Lemma 3.2 generates the category of $s$-balanced partitions, as one can check by a routine computation.
\end{proof}

Consider now the complex reflection group $H_n^s=\mathbb Z_s\wr S_n$, consisting of monomial matrices having the $s$-roots of unity as nonzero entries. Observe that we have $PH_n^{(s)}=H_n^s/\mathbb T$.

We have the following $s$-analogue of Theorem 2.6.

\begin{theorem}
$PH_n^{(s)}=PH_n^s$.
\end{theorem}

\begin{proof}
Observe first that this statement holds indeed at $s=2$, because here we have $H_n^{(2)}=H_n^2=H_n$. This statement holds as well at $s=\infty$, due to Theorem 2.6.

In the general case, this follows by adapting the proof of Theorem 2.6. Observe first that from $H_n^{(s)}\subset H_n^*$ we get $PH_n^{(s)}\subset PH_n^*=PK_n$, so $PH_n^{(s)}$ is indeed a classical group.

In order to compute this group, consider the following diagram:
$$\begin{matrix}
H_n^s&\subset&U_n^+\\
\\
\cup&&\cup\\
\\
S_n&\subset&H_n^{(s)}
\end{matrix}$$

The corresponding sets of partitions, as in the proof of Theorem 2.6, are as follows:
$$\begin{matrix}
E_h^s(2k,2l)&\supset&E_o(2k,2l)\\ 
\\ 
\cup&&\cup\\
\\
P(2k,2l)&\supset&E_h^s(2k,2l)
\end{matrix}$$

Indeed, the bottom left set is the good one, thanks to Theorem 1.3. The bottom right set is also the good one, thanks to Theorem 3.3. For the top right set, this was already discussed in the proof of Theorem 2.6. As for the top left set, this follows either from the results in \cite{bb+}, \cite{bv1} regarding the free version $H_n^{s+}$, after ``adding a crossing'', or from the $s=\infty$ computation in the proof of Theorem 2.6. A direct proof here can be obtained as well, by working out the categorical interpretation of the various relations defining $H_n^s$.

Let us look now at the projective versions of the above quantum groups:
$$\begin{matrix}
PH_n^s&\subset&PU_n^+\\
\\
\cup&&\cup\\
\\
PH_n&\subset&PH_n^{(s)}
\end{matrix}$$

As in the proof of Theorem 2.6, we are in the situation where we have two quantum subgroups having the same diagrams, and we conclude that we have $PH_n^{(s)}=PH_n^s$.
\end{proof}

\section{The higher hyperoctahedral series}

In this section we introduce a second one-parameter series of quantum groups, $H_n^{[s]}$ with $s\in\{2,3,\ldots,\infty\}$, having as main particular case the group $H_n^{[2]}=H_n$.

\begin{definition}
$C(H_n^{[s]})$ is the universal $C^*$-algebra generated by $n^2$ self-adjoint variables $u_{ij}$, subject to the following relations:
\begin{enumerate}
\item Orthogonality: $uu^t=u^tu=1$, where $u=(u_{ij})$ and $u^t=(u_{ji})$.

\item Ultracubic relations: $acb=0$, for any $a\neq b$ on the same row or column of $u$.

\item $s$-mixing relation: $abab\ldots=baba\ldots$ (length $s$ words), for any $a,b\in\{u_{ij}\}$.
\end{enumerate}
\end{definition}

The fact that $H_n^{[s]}$ is indeed a quantum group follows from the elementary fact that the ultracubic relations are of ``Hopf type'', i.e. that they allow the construction of the Hopf algebra maps $\Delta,\varepsilon,S$. This can be checked indeed by a routine computation.

Our first task is to compare the defining relations for $H_n^{[s]}$ with those for $H_n^{(s)}$. In order to deal at the same time with the cubic and ultracubic relations, it is convenient to use a statement regarding a certain unifying notion, of ``$k$-cubic'' relations.

\begin{lemma}
For a compact quantum group $G\subset O_n^+$, the following are equivalent:
\begin{enumerate}

\item The basic coordinates $u_{ij}$ satisfy the $k$-cubic relations $ac_1\ldots c_kb=0$, for any $a\neq b$ on the same row or column of $u$, and for any $c_1,\ldots,c_k$.

\item We have $T_p\in End(u^{\otimes k+2})$, where $p=(1,1',k+2,k+2')(2,2')\ldots (k+1,k+1')$.
\end{enumerate}
\end{lemma}

\begin{proof}
According to the definition of $T_p$ given in section 1, the operator associated to the partition in the statement is given by the following formula:
$$T_p(e_a\otimes e_{c_1}\otimes\ldots\otimes e_{c_k}\otimes e_b)=\delta_{ab}e_a\otimes e_{c_1}\otimes\ldots\otimes e_{c_k}\otimes e_a$$

This gives the following formulae:
\begin{eqnarray*}
T_pu^{\otimes k+2}(e_a\otimes e_{c_1}\otimes\ldots\otimes e_{c_k}\otimes e_b)&=&\sum_{ij}e_i\otimes e_{j_1}\otimes...\otimes e_{j_k}\otimes e_i\otimes u_{ia}u_{j_1c_1}\ldots u_{j_kc_k}u_{ib}\\
u^{\otimes k+2}T_p(e_a\otimes e_{c_1}\otimes\ldots\otimes e_{c_k}\otimes e_b)&=&\delta_{ab}\sum_{ijl}e_i\otimes e_{j_1}\otimes...\otimes e_{j_k}\otimes e_l\otimes u_{ia}u_{j_1c_1}\ldots u_{j_kc_k}u_{la}
\end{eqnarray*}

Here the sums are over all indices $i,l$, and over all multi-indices $j=(j_1,\ldots,j_k)$. The identification of the right terms gives the equivalence in the statement.
\end{proof}

We can now establish the precise relationship between $H_n^{[s]}$ and $H_n^{(s)}$, and also show that no further series can appear in this way. 

\begin{proposition}
For $k\geq 1$ the $k$-cubic relations are all equivalent to the ultracubic relations, and they imply the cubic relations.
\end{proposition}

\begin{proof}
This follows from the following two observations:

(a) The $k$-cubic relations imply the $2k$-cubic relations. Indeed, one can connect two copies of the partition $p$ in Lemma 4.2, by gluing them with two semicircles in the middle, and the resulting partition is the one implementing the $2k$-cubic relations.

(b) The $k$-cubic relations imply the $(k-1)$-cubic relations. Indeed, by capping the partition $p$ in Lemma 4.2 with a semicircle at bottom right, we get a certain partition $p'\in P(k+2,k)$, and by rotating the upper right leg of this partition we get the partition $p''\in P(k+1,k+1)$ implementing the $(k-1)$-cubic relations.
\end{proof}

The above statement shows that replacing in Definition 4.1 the ultracubic condition by any of the $k$-cubic conditions, with $k\geq 2$, won't change the resulting quantum group. The other consequences of Proposition 4.3 are summarized as follows.

\begin{proposition}
The quantum groups $H_n^{[s]}$ have the following properties:
\begin{enumerate}
\item We have $H_n^{(s)}\subset H_n^{[s]}\subset H_n^+$.

\item At $s=2$ we have $H_n^{[2]}=H_n^{(s)}=H_n$.

\item At $s\geq 3$ we have $H_n^{(s)}\neq H_n^{[s]}$.
\end{enumerate}
\end{proposition}

\begin{proof}
All the assertions basically follow from Lemma 4.2:

(1) For the first inclusion, we need to show that half-commutation + cubic implies ultracubic, and this can be done by placing the half-commutation partition next to the cubic partition, then using 2 semicircle cappings in the middle.

The second inclusion follows from Proposition 4.3, because the ultracubic relations (1-cubic relations) imply the cubic relations (0-cubic relations).

(2) Observe first that at $s=2$ the $s$-commutation is the usual commutation $ab=ba$. Thus we are led here to the algebra generated by $n^2$ commuting self-adjoint variables satisfying the cubic condition, which is $C(H_n)$.

(3) Finally, $H_n^{(s)}\neq H_n^{[s]}$ will be a consequence of Theorem 4.5 below, because at $s\geq 3$ the half-commutation partition $p=(14)(25)(36)$ is $s$-balanced but not locally $s$-balanced.
\end{proof}

\begin{theorem}
$H_n^{[s]}$ is an easy quantum group, and its associated category is that of the ``locally $s$-balanced'' partitions, i.e. partitions having the property that each of their subpartitions (i.e. partitions obtained by removing certain blocks) are $s$-balanced.
\end{theorem}

\begin{proof}
As a first remark, at $s=2$ the locally $s$-balancing condition is automatic for a partition having blocks of even size, so we get indeed the category corresponding to $H_n$.

In the general case now, our first claim is that the locally $s$-balanced partitions from indeed a category. But this follows simply by adapting the argument in the proof of Theorem 3.3, just by adding ``locally'' everywhere.

It remains to prove that this category corresponds indeed to $H_n^{[s]}$. But this follows from Theorem 3.2 and from the fact that the partition generating the category of locally balanced partitions, namely $p=(1346)(25)$, is nothing but the one implementing the ultracubic relations, as one can check by a routine computation.
\end{proof}

\section{Classification: general strategy}

In this section and in the next one we advance on the classification work started in \cite{bsp}. We will prove that the easy quantum groups constructed so far are the only ones, modulo an hypotethical multi-parameter ``hyperoctahedral series'', unifying the series constructed in the previous sections, and still waiting to be constructed.

Let $G$ be an easy quantum group, with category of partitions denoted $P_g$. It follows from definitions that $P_g\cap NC$ is a category of noncrossing partitions, and by the results in section 1, this latter category must come from a certain free quantum group $K^+$. Observe that since $NC_k=P_g\cap NC$ is included into $P_g$, we have $G\subset K^+$. 

\begin{definition}
Associated to an easy quantum group $G$ is the easy group $K$ given by the equality of categories $P_g\cap NC=NC_k$.
\end{definition}

According now to the easy group classification in Theorem 1.3, there are 6 cases to be investigated. We will split the study into two parts: 5 cases will be investigated in section 6, and the remaining case, $K=H_n$, will be eventually left open.

The point with this splitting comes from the following question: do we have $K\subset G$? In the reminder of this section we will try to answer this question.

We begin with the technical lemma, valid in the general case. Let $\Lambda_g,\Lambda_k\subset\mathbb N$ be the set of the possible sizes of blocks of elements of $P_g,NC_k$.

\begin{lemma}
Let $G,K$ be as above.
\begin{enumerate}
\item $\Lambda_k\subset\Lambda_g\subset \Lambda_k\cup(\Lambda_k-1)$.

\item $1\in\Lambda_g$ implies $1\in\Lambda_k$.

\item If $NC_k$ is even, so is $P_g$.
\end{enumerate}
\end{lemma}

\begin{proof}
We will heavily use the various abstract notions and results in \cite{bsp}.

(1) Here the first inclusion follows from $NC_k\subset P_g$. As for the second inclusion, this is equivalent to the following statement: ``If $b$ is a block of a partition $p\in P_g$, then there exists a certain block $b'$ of a certain partition $p'\in P_g\cap NC$, having size $\# b$ or $\# b-1$''. 

But this latter statement follows by using the ``capping'' method in \cite{bsp}. Indeed, we can cap $p$ with semicircles, as for $b$ to remain unchanged, and we end up with a certain partition $p'$ consisting of $b$ and of some extra points, at most one point between any two legs of $b$, which might be connected or not. Note that the semicircle capping being a categorical operation, this partition $p'$ remains in $P_g$.

Now by further capping $p'$ with semicircles, as to get rid of the extra points, the size of $b$ can only increase, and we end up with a one-block partition having size at least that of $b$. This one-block partition is obviously noncrossing, and by capping it again with semicircles we can reduce the number of legs up to $\# b$ or $\# b-1$, and we are done.

(2) The condition $1\in\Lambda_g$ means that there exists $p\in P_g$ having a singleton. By capping $p$ with semicircles outside this singleton, we can obtain a singleton, or a double singleton. Since both these partitions are noncrossing, and have a singleton, we are done. 

(3) Indeed, assume that $P_g$ is not even, and consider a partition $p\in P_g$ having an odd number of legs. By capping $p$ with enough semicircles we can arrange for ending up with a singleton, and since this singleton is by definition in $P_g\cap NC$, we are done.
\end{proof}

We are now in position of splitting the classification. We have the following key result.

\begin{proposition}
Let $G,K$ be as above.
\begin{enumerate}
\item If $K\neq H_n$ then $K\subset G\subset K^+$.

\item If $K=H_n$ then $S_n'\subset G\subset H_n^+$.
\end{enumerate}
\end{proposition}

\begin{proof}
We recall that the inclusion $G\subset K^+$ follows from definitions. For the other inclusion, we have 6 cases, depending on the exact value of the easy group $K$:

(1.1) $K=O_n$. Here $\Lambda_k=\{2\}$, so by Lemma 5.2 (1) we get $\{2\}\subset\Lambda_g\subset\{1,2\}$. Moreover, from Lemma 5.2 (2), we get $\Lambda_g=\{2\}$. Thus $P_g\subset P_o$, which gives $O_n\subset G$.

(1.2) $K=S_n$. Here there is nothing to prove, since $S_n\subset G$ by definition.

(1.3) $K=B_n$. Here $\Lambda_k=\{1,2\}$, so by Lemma 5.2 (1) we get $\Lambda_g=\{1,2\}$. Thus we have $P_g\subset P_b$, which gives $B_n\subset G$.

(1.4) $K=S_n'$. Here we have $P_g\subset P_s$ by definition, and by using Lemma 5.2 (3) we deduce that we have $P_g\subset P_{s'}$, which gives $S_n'\subset G$.

(1.5) $K=B_n'$.  Here we have $\Lambda=\{1,2\}$, so by Lemma 5.2 (1) we get $\Lambda_g=\{1,2\}$. This gives $P_g\subset P_b$, and by Lemma 5.2 (3) we get $P_g\subset P_{b'}$, which gives $B_n'\subset G$.

(2) $K=H_n$. Here we have $P_g\subset P_s$ by definition, and by using Lemma 5.2 (3) we deduce that we have $P_g\subset P_{s'}$, which gives $S_n'\subset G$.
\end{proof}

With a little more care, one can prove that the easy group $K$ in the above statement (1) is nothing but the ``classical version'' of $G$, obtained as dual object to the commutative Hopf algebra $C(G)/I$, where $I\subset C(G)$ is the commutator ideal.

Observe also that the above statement (2) cannot be improved. The point is that for the quantum group $H_n^{(s)}$ with $s$ odd we have $K=H_n$, and $K\not\subset G$.

\section{The non-hyperoctahedral case}

In this section we classify the easy quantum groups, under the non-hyperoctahedral assumption $K\neq H_n$. Here $K$ is as usual the easy group from Definition 5.1.

We know from Proposition 5.3 that our easy quantum group $G$ appears as an intermediate quantum group, $K\subset G\subset K^+$. In order to classify these intermediate quantum groups, we use the method in \cite{bv2}, where the problem was solved in the case $G=O_n$. For uniformity reasons, we will include as well the case $G=O_n$ in our study.

We will need a number of technical ingredients.

\begin{definition}
Let $p\in P(k,l)$ be a partition, with the points counted modulo $k+l$, counterclockwise starting from bottom left.
\begin{enumerate}
\item We call semicircle capping of $p$ any partition obtained from $p$ by connecting with a semicircle a pair of consecutive neighbors.

\item We call singleton capping of $p$ any partition obtained from $p$ by capping one of its legs with a singleton.

\item We call doubleton capping of $p$ any partition obtained from $p$ by capping two of its legs with singletons.
\end{enumerate}
\end{definition}

In other words, the semicircle, singleton and doubleton cappings are elementary operations on partitions, which lower the total number of legs by $2,1,2$ respectively. Observe that there are $k+l$ possibilities for placing the semicircle or the singleton, and $(k+l)(k+l-1)/2$ possibilities for placing the double singleton. Observe also that in the case of 2 particular ``semicircle cappings'', namely those at left or at right, the semicircle in question is rather a vertical bar; but we will still call it semicircle.

The various cappings of $p$ will be generically denoted $p'$.

Consider now the $5+5+1=11$ categories of partitions $P_x,NC_x,E_x$, with $x=o,s,b,s',b'$ described in sections 1 and 2. We have the following technical lemma. 

\begin{lemma}
Let $p$ be a partition, having $j$ legs.
\begin{enumerate}
\item If $p\in P_o-E_o$ and $j>4$, there exists a semicircle capping $p'\in P_o-E_o$.

\item If $p\in E_o-NC_o$ and $j>6$, there exists a semicircle capping $p'\in E_o-NC_o$.

\item If $p\in P_s-NC_s$ and $j>4$, there exists a singleton capping $p'\in P_s-NC_s$.

\item If $p\in P_b-NC_b$ and $j>4$, there exists a singleton capping $p'\in P_b-NC_b$.

\item If $p\in P_{s'}-NC_{s'}$ and $j>4$, there exists a doubleton capping $p'\in P_{s'}-NC_{s'}$.

\item If $p\in P_{b'}-NC_{b'}$ and $j>4$, there exists a doubleton capping $p'\in P_{b'}-NC_{b'}$.
\end{enumerate}
\end{lemma}

\begin{proof}
We write $p\in P(k,l)$, so that the number of legs is $j=k+l$. In the cases where our partition is a pairing, we use as well the number of strings, $s=j/2$.

Let us agree that all partitions are drawn as to have a minimal number of crossings.

We use the same idea for all the proofs, namely to ``isolate'' a block of $p$ having a crossing, or an odd number of crossings, then to ``cap'' $p$ as in the statement, as for this block to remain crossing, or with an odd number of crossings. Here we use of course the observation that the ``balancing'' condition which defines the categories $E_o,E_h$ can be interpreted as saying that each block has an even number of crossings, when the picture of the partition is drawn such that this number of crossings is minimal.

(1) The assumption $p\notin E_o$ means that $p$ has certain strings having an odd number of crossings. We fix such an ``odd'' string, and we try to cap $p$, as for this string to remain odd in the resulting partition $p'$. An examination of all possible pictures shows that this is possible, provided that our partition has $s>2$ strings, and we are done.

(2) The assumption $p\notin NC_o$ means that $p$ has certain crossing strings. We fix such a pair of crossing strings, and we try to cap $p$, as for these strings to remain crossing in $p'$. Once again, an examination of all possible pictures shows that this is possible, provided that our partition has $s>3$ strings, and we are done.

(3) Indeed, since $p$ is crossing, we can choose two of its blocks which are intersecting. If there are some other blocks left, we can cap one of their legs with a singleton, and we are done. If not, this means that our two blocks have a total of $j'\geq j>4$ legs, so at least one of them has $j''>2$ legs. One of these $j''$ legs can always be capped with a singleton, as for the capped partition to remain crossing, and we are done.

(4) Here we can simply cap with a singleton, as in (3).

(5,6) Here we can cap with a doubleton, by proceeding twice as in (3).
\end{proof}

For a collection of subsets $X(k,l)\subset P(k,l)$ we denote by $<X>\subset P$ the category of partitions generated by $X$. In other words, the elements of $<X>$ come from those of $X$ via the categorical operations for the categories of partitions, which are the vertical and horizontal concatenation and the upside-down turning. See \cite{bsp}.

\begin{lemma}
Let $p$ be a partition.
\begin{enumerate}
\item If $p\in P_o-E_o$ then $<p,NC_o>=P_o$.

\item If $p\in E_o-NC_o$ then $<p,NC_o>=E_o$.

\item If $p\in P_s-NC_s$ then $<p,NC_s>=P_s$.

\item If $p\in P_b-NC_b$ then $<p,NC_b>=P_b$.

\item If $p\in P_{s'}-NC_{s'}$ then $<p,NC_{s'}>=P_{s'}$.

\item If $p\in P_{b'}-NC_{b'}$ then $<p,NC_{b'}>=P_{b'}$.
\end{enumerate}
\end{lemma}

\begin{proof}
We use Lemma 6.2, together with the observation that the ``capping partition'' appearing there is always in the good category. 

That is, we use the fact that the semicircle is in $NC_o,NC_{s'}$, the singleton is in $NC_s,NC_b$, and the doubleton is in $NC_{b'}$. This observation tells us that, in each of the cases under consideration, the category to be computed can only decrease when replacing $p$ by one of its cappings $p'$. Indeed, for the singleton and doubleton cappings this is clear from definitions, and for the semicircle capping this is clear as well from definitions, unless in the case where the ``capping semicircle'' is actually a ``bar'' added at left or at right, where we can use a categorical rotation operation as in \cite{bsp}.

(1,2) These assertions can be proved by recurrence on the number of strings, $s=(k+l)/2$. Indeed, by using Lemma 6.2 (1,2), for $s>3$ we have a descent procedure $s\to s-1$, and this leads to the situation $s\in\{1,2,3\}$, where the statement is clear.

(3) We can proceed by recurrence on the number of legs of $p$. If the number of legs is $j=4$, then $p$ is a basic crossing, and we have $<p>=P_s$. If the number of legs is $j>4$ we can apply Lemma 6.2 (3), and the result follows from $<p>\supset <p'>=P_s$.

(4,5,6) This is similar to the proof of (1,3,2), by using Lemma 6.2 (4,5,6).
\end{proof}

\begin{lemma}
Let $p$ be a partition.
\begin{enumerate}
\item If $p\in P_o$ then $<p,NC_o>\in\{P_o,E_o,NC_o\}$.

\item If $p\in P_s$ then $<p,NC_s>\in\{P_s,NC_s\}$.

\item If $p\in P_b$ then $<p,NC_b>\in\{P_b,NC_b\}$.

\item If $p\in P_{s'}$ then $<p,NC_{s'}>\in\{P_{s'},NC_{s'}\}$.

\item If $p\in P_{b'}$ then $<p,NC_{b'}>\in\{P_{b'},NC_{b'}\}$.
\end{enumerate}
\end{lemma}

\begin{proof}
This follows by rearranging the various technical results in Lemma 6.3.
\end{proof}

We are now in position of stating the main result in this paper. Let us call ``non-hyperoctahedral'' any easy quantum group $G$ such that $K\neq H_n$.

\begin{theorem}
There are exactly $11$ non-hyperoctahedral easy quantum groups, namely:
\begin{enumerate}
\item $O_n,O_n^*,O_n^+$: the orthogonal quantum groups.

\item $S_n,S_n^+$: the symmetric quantum groups.

\item $B_n,B_n^+$: the bistochastic quantum groups.

\item $S_n',S_n'^+$: the modified symmetric quantum groups.

\item $B_n',B_n'^+$: the modified bistochastic quantum groups.
\end{enumerate}
\end{theorem}

\begin{proof}
We know from Proposition 5.3 that what we have to do is to classify the easy quantum groups satisfying $K\subset G\subset K^+$. More precisely, we have to prove that for $K=S_n,B_n,S_n',B_n'$ there is no such partial liberation, and that for $K=O_n$ there is only one partial liberation, namely the above-mentioned quantum group $K^*$. But this follows from Lemma 6.4, via the Tannakian results in \cite{bsp}.
\end{proof}

As for the classification in the hyperoctahedral case, this seems to be a quite difficult problem, that we have to leave open.

\section{Laws of characters}

In this section we discuss the computation of the asymptotic law of the fundamental character $\chi={\rm Tr}(u)$, and of its truncated versions $\chi_t=\sum_{i=1}^{[tn]}u_{ii}$ with $t\in (0,1]$.

These computations, which might seem quite technical, are in fact of great relevance in the general context of representation theory. Given a compact group $G\subset U_n$, or more generally a compact quantum group $G\subset U_n^+$, the main representation theory problem is to classify the irreducible representations of $G$. By the Peter-Weyl theory these irreducible representations appear into the tensor powers $u^{\otimes k}$ of the fundamental representation, and they can be in fact identified with the minimal projections of the algebra $End(u^{\otimes k})$.

The exact computation of $End(u^{\otimes k})$ is in general a quite difficult problem. However, an easier question, whose answer is in general extremely useful, concerns the computation of the dimension of this algebra. Now since this dimension can be simply obtained by integrating $\chi^{2k}$, we are led to the fundamental problem of computing the law of $\chi$.

In the quantum group context, the lone computation of the law of $\chi$, and the comparison with the classical results, can be quite puzzling. The problem appears for instance with $S_n$ and $S_n^+$, where the law of $\chi$ is respectively Poisson with $n\to\infty$, and free Poisson with $n\geq 4$. The lack of symmetry in this fundamental computation was conceptually understood in \cite{bco}, where it was shown that the correct invariant to look at is the law of the truncated character $\chi_t$, with $t\in (0,1]$. So, our starting definition will be as follows.

\begin{definition}
Associated to an easy quantum group $G\subset U_n^+$ is the truncated character
$$\chi_t=\sum_{i=1}^{[tn]}u_{ii}$$
where $u=(u_{ij})$ is the matrix of standard coordinates, and $t\in (0,1]$.
\end{definition}

Let us recall now some basic results from \cite{bsp}. Let $G$ be an easy quantum group, and denote by $D_k\subset P(0,k)$ the corresponding sets of diagrams, having no upper points. We define the Gram matrix to be $G_{kn}(p,q)=n^{b(p\vee q)}$, where $b(.)$ is the number of blocks. The Weingarten matrix is by definition its inverse, $W_{kn}=G_{kn}^{-1}$. In order for this inverse to exist, $n$ has to be big enough, and the assumption $n\geq k$ is sufficient. In the general case the notion of quasi-inverse must be used, see \cite{cma} for a detailed discussion here.

\begin{theorem}
The Haar integration over $G$ is given by
$$\int_Gu_{i_1j_1}\ldots u_{i_kj_k}\,du=\sum_{p,q\in D_k}\delta_p(i)\delta_q(j)W_{kn}(p,q)$$
where the $\delta$ symbols are $0$ or $1$, depending on whether the indices fit or not.
\end{theorem}

\begin{proof}
This is proved in \cite{bsp}, the idea being that the integrals on the left, taken altogether, form the orthogonal projection on $Fix(u^{\otimes k})=span(D_k)$.
\end{proof}

The Weingarten formula is particularly effective in the classical and free cases, where complete computations were performed in \cite{bsp}. Let us record here the following result.

\begin{theorem}
The asymptotic law of $\chi_t=\sum_{i=1}^{[tn]}u_{ii}$ with $t\in (0,1]$ is as follows:
\begin{enumerate}
\item For $O_n,S_n,H_n,B_n$ we get the Gaussian, Poisson, Bessel and shifted Gaussian laws, which form convolution semigroups.

\item For $O_n^+,S_n^+,H_n^+,B_n^+$ we get the semicircular, free Poisson, free Bessel and shifted semicircular laws, which form free convolution semigroups.

\item For $S_n',B_n',S_n'^+,B_n'^+$ we get symmetrized versions of the laws for $S_n,B_n,S_n^+,B_n^+$, which do not form classical or free convolution semigroups.
\end{enumerate}
\end{theorem}

\begin{proof}
This is proved in \cite{bsp}, by using the Weingarten formula and cumulants. Note that the semigroups in (1) and (2) are in Bercovici-Pata bijection \cite{bpa}. 
\end{proof}

We should mention that the measures in (3), while not forming semigroups due to the canonical copy of $\mathbb Z_2$, which produces a ``correllation'', are very close to forming some kind of semigroup. We intend to come back to this question in our next papers \cite{ez2}, \cite{ez3}.

In the remaining cases, the Weingarten formula is less effective, because the counting of partitions and of their blocks is a quite delicate task. In the case of half-liberations and of the hyperoctahedral series we will use instead the projective versions computed in the previous sections, which reduce the problem to a classical computation.

Let us begin with the following definition.

\begin{definition}
We use the following complex probability measures:
\begin{enumerate}
\item The complex Gaussian law of parameter $t>0$ is the law of $x+iy$, where $x,y$ are Gaussian variables of parameter $t$, independent. 

\item The $s$-Bessel law of parameter $t>0$ is the law of $\sum_{r=1}^se^{2\pi ir/s}x_i$, where $x_1,\ldots,x_s$ are Poisson variables of parameter $t/s$, independent.
\end{enumerate}
\end{definition}

The complex Gaussian laws are well-known to form a convolution semigroup. The same holds for the $s$-Bessel laws, and we refer to \cite{bb+} for a complete discussion here. Let us just mention that the ``Bessel'' terminology comes from the fact that at $s=2$, the density of the corresponding discrete measure on $\mathbb R$ is given by a Bessel function of the first kind.

\begin{definition}
Given a complex probability measure $\mu$, we call ``squeezed version'' of it the law of $\sqrt{zz^*}$, where $z$ follows the law $\mu$. 
\end{definition}

This law doesn't depend of course on the choice of $z$.

As an example, the squeezed version of the complex Gaussian law of parameter 1 is the Rayleigh law. This is because with $z=x+iy$ we have $zz^*=x^2+y^2$.

Another interesting example, of key relevance in free probability, is the fact that the squeezed version of Voiculescu's circular law is Wigner's semicircle law. See e.g. \cite{nsp}. 

\begin{theorem}
The asymptotic law of $\chi_t=\sum_{i=1}^{[tn]}u_{ii}$ with $t\in (0,1]$ is as follows:
\begin{enumerate}
\item For $O_n^*$ we get the squeezed complex Gaussian semigroup.

\item For $H_n^{(s)}$ we get the squeezed $s$-Bessel semigroup. 
\end{enumerate}
\end{theorem}

\begin{proof}
The Weingarten formula shows that the odd moments of the variables in the statement are all 0, so all computations actually take place over the projective versions. With this remark in hand, the results simply follow from the well-known fact that $\chi_t$ is asymptotically complex Gaussian for $U_n$, and $s$-Bessel for $H_n^s$. See \cite{bb+}.
\end{proof}

The squeezed $s$-Bessel laws seem to have a quite interesting combinatorics, but it is beyond the purposes of this paper to get into this subject. We would like however to present one such combinatorial statement, in the simplest case, $s=\infty$ and $t=1$.

\begin{proposition}
The asymptotic even moments of the character $\chi\in C(H_n^*)$ satisfy
$$c_k=\sum_{s=0}^{k-1}\begin{pmatrix}k\cr s\end{pmatrix}\begin{pmatrix}k-1\\ s\end{pmatrix}c_s$$
and are equal to the number of games of simple patience with $n$ cards.
\end{proposition}

\begin{proof}
This follows from Theorem 7.6, but we will present below a direct proof, that we found at an early stage of this work. According to the general theory, the numbers in the statement are given by $c_k=\# E_h(2k)$, i.e. they count the partitions of $\{1,\ldots,2k\}$ having the property that each block has the same number of odd and even legs. 

It is convenient to do the following manipulation: we keep the sequence of odd legs fixed, and we pull downwards the sequence of even legs. In this way, $E_h(2k)$ becomes the set of partitions between an upper and a lower sequence of $k$ points, such that each block is ``balanced'', in the sense that it has the same number of upper and lower legs. 

Now observe that these partitions can be obtained as follows: (a) pick a number $r\in\{1,\ldots,n\}$, (b) connect the first point on the upper line to some $r-1$ other points on the upper line, (c) choose $r$ points on the lower line, and connect them to the already connected upper $r$ points, (d) connect the remaining $k-r$ upper points to the remaining $k-r$ lower points, by means of a balanced partition. 

With $s=k-r$ this gives the formula in the statement. As for the patience game interpretation, see Aldous and Diaconis \cite{adi} and Sloane's comments in \cite{slo} regarding the sequence A023998, which is the sequence of moments of $\chi$.
\end{proof}

Finally, let us mention that for the higher hyperoctahedral quantum group $H_n^{[s]}$, our standard methods simply don't work. We don't know if this quantum group produces or not the squeezed version of some ``known'' semigroup.

\section{Concluding remarks}

We have seen in this paper that the easy quantum groups consist in principle of 6 groups, their free versions, 2 half-liberations, and one infinite series, still waiting to be constructed. The construction of this hypothetical multi-parameter ``hyperoctahedral series'', and the continuation and finishing of our classification work, are of course the main two questions that we would like to address here.

The situation here, which appears to be unexpectedly complex, reminds a bit the algebraic difficulty and subtlety of the usual complex reflection groups \cite{bmr}.

At the level of applications now, as explained in the introduction, we intend to use the easy quantum group list that we have so far as an ``input'' for a number of representation theory and probability considerations, based on our belief that ``any result which holds for $S_n,O_n$ should have a suitable extension to all easy quantum groups''. 

In addition, in the non-easy case, there are of course of big number of results, classical or even free, having something to do with ``diagrams'' and with the easy quantum group technology in general, and that might fall one day into an extension of our formalism.

Here is a list of topics, waiting to be developed:

\begin{enumerate}
\item De Finetti theorems. These are available for $S_n,O_n$ from the book \cite{kal}, for $S_n^+$ from \cite{ksp} then \cite{cu1}, and for $O_n^+$ from \cite{cu2}. We intend to develop a global approach to the problem, by using easy quantum groups, in our forthcoming paper \cite{ez2}.

\item Eigenvalue computations. The key results of Diaconis and Shahshahani in \cite{dsh} concerning $S_n,O_n$ can be obtained as well by using  Weingarten functions and cumulants, and an extension to all easy quantum groups is in preparation \cite{ez3}.

\item Invariant theory. The groups $S_n,O_n$ and their versions $S_n^+,O_n^+,O_n^*$ have served as a guiding example for the study of many invariants, see \cite{bco}, \cite{bv1}, \cite{bv2}, \cite{cma}, \cite{csn}, \cite{nov}. Some of these results are expected to extend to all easy quantum groups.

\item Geometric aspects. The groups $S_n,O_n$ and their free versions $S_n^+,O_n^+$ were involved as well in many other ``classical vs. free'' considerations. Let us mention here the Poisson boundary results in \cite{vve}, and the quantum isometry groups in \cite{bgo}. Once again, the easy quantum groups can lead to some new results here.

\item Generalizations. One interesting question would be to understand the twisting and deformation of the easy quantum groups, say with the objective of extending our formalism to the $S^2\neq id$ case, via monoidal equivalence \cite{bdv}. Another question is whether the half-liberation operation can be applied to locally compact real algebraic groups $G\subset M_n(\mathbb R)$, as to fit into the general axioms in \cite{kva}.
\end{enumerate}

In addition to these questions, one basic problem is to classify the intermediate quantum groups $K\subset G\subset K^+$, where $K$ is a fixed easy group. This looks like a quite difficult question. However, a ray of light comes from a conjecture in \cite{qpg}, stating that there is no intermediate quantum group $S_n\subset G\subset S_n^+$. This is actually a quite subtle question, whose study seems to lead straight into the core of the ``non-easy'' problematics.

\end{document}